\documentclass[12pt, reqno]{amsart}
\usepackage{geometry}   
\usepackage{enumitem}
\usepackage{appendix}

\usepackage[colorlinks,citecolor = red, linkcolor=blue,hyperindex]{hyperref}
\usepackage{euscript,eufrak,verbatim, mathrsfs}
\usepackage[psamsfonts]{amssymb}
\usepackage{bbm}
\usepackage{graphicx}
\usepackage{float}
\usepackage{float, tikz}
\usepackage{subcaption}
\usetikzlibrary{calc}
\usetikzlibrary{arrows.meta,patterns}
\usepackage{pgfplots}
\usepgfplotslibrary{fillbetween}
\pgfplotsset{compat=1.18}

\usepackage{extpfeil}

\usepackage[all, cmtip]{xy}

\usepackage{upref, xcolor, dsfont}
\usepackage{amsfonts,amsmath,amstext,amsbsy, amsopn,amsthm}

\usepackage{url}

\usepackage{mathtools}
\usepackage{bookmark}

\usepackage{euscript}
\usepackage{times}
\usepackage{type1cm}        
\usepackage{multicol}        
\usepackage[bottom]{footmisc}
\newcommand{\normmm}[1]{{\left\vert\kern-0.25ex\left\vert\kern-0.25ex\left\vert #1 
		\right\vert\kern-0.25ex\right\vert\kern-0.25ex\right\vert}}

\newtheorem{theorem}{Theorem}[section]
\newtheorem*{theorem*}{Theorem B}
\newtheorem{lemma}[theorem]{Lemma}

\newtheorem{proposition}[theorem]{Proposition}

\newtheorem*{observation*}{Observation}

\newtheorem*{assumption*}{Assumption}
\newtheorem*{question*}{Question}

\newtheorem*{Theorem K}{Kahane's theorem}

\theoremstyle{definition}

\newtheorem*{definition*}{Definition}

\theoremstyle{remark}
\newtheorem{remark}{Remark}[section]
\newtheorem*{remark*}{Remark}

\newtheorem{problem}{Problem}

\newcommand{\R}{\mathbb{R}}

\newcommand{\Z}{\mathbb{Z}}

\newcommand{\E}{\mathbb{E}}
\newcommand{\T}{\mathbb{T}}
\newcommand{\PP}{\mathbb{P}}

\newcommand{\supp}{\mathrm{supp}}

\newcommand{\RC}{\textnormal{RC}}

\newcommand{\an}{\text{\quad and \quad}}

\newcommand{\indi}{\mathds{1}}

\numberwithin{equation}{section}

\begin{document}
\title{Salem properties of Dvoretzky random coverings}

\author{Yukun Chen}
\address{Yukun Chen: School of Mathematics and Statistics, Wuhan University, Wuhan 430072, China}
\email{yukunchen@whu.edu.cn}

\author{Xiangdi Fu}
\address{Xiangdi Fu: School of Fundamental Physics and Mathematical Sciences, HIAS, University of Chinese Academy of Sciences, Hangzhou, 310024, China}
\email{xdfu@ucas.ac.cn}

\author{Zhaofeng Lin}
\address{Zhaofeng Lin: School of Fundamental Physics and Mathematical Sciences, HIAS, University of Chinese Academy of Sciences, Hangzhou, 310024, China}
\email{linzhaofeng@ucas.ac.cn}

\author{Yanqi Qiu}
\address{Yanqi Qiu: School of Fundamental Physics and Mathematical Sciences, HIAS, University of Chinese Academy of Sciences, Hangzhou, 310024, China}
\email{yanqiqiu@ucas.ac.cn}


	\begin{abstract}
	 We establish the Salem properties for the uncovered sets in the celebrated Dvoretzky random coverings of the unit circle. 
	\end{abstract}


	\maketitle

	\setcounter{tocdepth}{2}

	\setcounter{tocdepth}{0}
	\setcounter{equation}{0}




\section{Introduction}
In this paper, we establish the Salem properties for the uncovered sets in the celebrated Dvoretzky random coverings of the unit circle. Our proof relies on a {\it translation-cancellation technique}, combined with the {\it smallness of  $L^1$-modulus of continuity} and the {\it smallness of mean-oscillation} of sample paths,  all incorporated into the vector-valued martingale framework for Fourier decay.   
These new ingredients enable us, in particular, to handle multiplicative chaos measures associated with $T$-martingales that lack certain forms of {\it weak spatial independence} (see \S \ref{sec-new-ingre} for more explanations of weak spatial independence).   Our approach is elementary and the paper is written in a self-contained way.

We mention that,  due to the lack of the weak spatial independence in the model of Dvoretzky random coverings, neither  the core arguments in \cite{CHQW24} nor those in \cite{LQT25II} can be applied in the current situation; meanwhile,   due to the lack of  the smallness of  $L^1$-modulus of continuity and the smallness of mean-oscillation of sample paths,  the new ingredients in this paper can not be applied to the models in \cite{CHQW24, LQT25II}.

\subsection{Dvoretzky random coverings}
We identify the unit circle \(\mathbb{T} = \mathbb{R}/\mathbb{Z}\) with the unit interval \([0,1)\). Throughout this paper, we denote by \( \ell = (\ell_k)_{k\ge 1}\) a {\it non-increasing} sequence satisfying
$
0<\ell_k < 1$ and $\sum_{k=1}^\infty \ell_k = \infty$,  and let \(\omega = (\omega_k)_{k\ge 1}\) be  i.i.d. random variables uniformly distributed on \(\mathbb{T}\). We denote by $I_k = I_k(\omega)$ the open random arcs 
\(
I_k = (\omega_k, \omega_k + \ell_k)  \subset \T  = \R/\Z, 
\)
and by
\[ E_\ell:= \limsup I_k\]
the random set consisting of all points on $\T$ that are covered by $I_k$ infinitely often. As an immediate consequence of the Borel--Cantelli lemma, the complement of $E_\ell$ has Lebesgue measure zero almost surely.
The celebrated Dvoretzky covering problem \cite{Dvo56} asks for which sequences $\ell$ does one have
\[ 
\PP \big(E_\ell=\T \big) = 1,
\] meaning that, almost surely, \emph{every} point on $\T$ is covered infinitely often by the arcs $I_k$.  

Dvoretzky covering problem attracted the attention of P. L\'evy, J.-P. Kahane, P. Erd\H{o}s, P. Billard, and B. Mandelbrot (see \cite[Chapter 11]{Kah00}), and the complete solution was given by L. Shepp \cite{She72}: 
\begin{align}\label{eq-Shepp-cond}
E_\ell = \T\,\, \text{a.s.} \,\, \text{if and only if}\,\,\sum_{k=1}^\infty k^{-2}\exp\big(\ell_1 + \cdots +  \ell_k \big)=\infty.
\end{align}
We should mention that the canonical case $\ell_k= \alpha/k$ with $\alpha> 0$, was established before Shepp's work. The result for $\alpha>1$ was proved by Kahane \cite{Kah59}, and that for $\alpha<1$ was proved by Billard \cite{Bil65}. Billard also proved that for $\alpha = 1$, the set $\T \setminus E_\ell$ is at most countable, and then Mandelbrot \cite{Man72} and Orey (unpublished) independently showed that $E_\ell = \T$ in this critical case. We refer the reader to \cite{Kah00} for additional historical remarks, and to \cite{BarFan05,Fan02, FanWu04, FFS85, JonSte08, Kah85} for further studies on this topic.

\subsection{Main result}
In this paper, we are interested in the complement of \(E_\ell\), which we denote by 
\begin{align}\label{def-K-l}
K_\ell := \T \setminus E_\ell  =\T\setminus \limsup_k I_k. 
\end{align}
Conventionally, the set \(K_\ell\) is usually referred to as the {\it uncovered set} (more precisely, the set of points not covered infinitely many times). By Kolmogorov's zero-one law and Shepp's criterion \eqref{eq-Shepp-cond},  one has $\PP(K_\ell\neq \emptyset)>0$ if and only if $\PP(K_\ell\neq \emptyset)=1$ if and only if
\begin{align}\label{eq-non-Shepp}
 \sum_{k=1}^\infty k^{-2}\exp\big(\ell_1 + \cdots +  \ell_k \big)<\infty.
\end{align}
Therefore, the study of $K_\ell$ naturally assumes \eqref{eq-non-Shepp} (otherwise $K_\ell$ is empty almost surely). In the non-empty case, $K_\ell$ exhibits a certain fractal structure \cite{Man82}, and it is always small in both the sense of Lebesgue measure and Baire category. It is natural to use several fractal dimensions to describe its size and geometric properties. In particular, Kahane \cite[Theorem 4, page 160]{Kah85} determined the Hausdorff dimension of \(K_\ell\), and the following quantity plays an important role in his work: for the sequence $\ell = (\ell_k)_{k\ge 1}$, define
\begin{align}\label{def-D-l}
D_\ell := \limsup_{k\to\infty} \frac{\ell_1   + \cdots + \ell_k}{\log k}. 
\end{align}
It is clear that $D_\ell <1$ implies \eqref{eq-non-Shepp}. Moreover, a weaker converse holds (see Appendix \ref{appendix:non-shepp-D_l<=1}):
\begin{align*}
\text{$\eqref{eq-non-Shepp}\Longrightarrow D_\ell\leq 1$}
\end{align*}

\begin{Theorem K}\label{thm-Kahane}
 Assuming \eqref{eq-non-Shepp}, the uncovered set $K_\ell$ is almost surely non-empty with Hausdorff dimension $$\dim_{\mathcal H} K_\ell = 1-D_\ell.$$
\end{Theorem K}

\begin{remark}\label{rem-0dim}
	In Kahane's original statement of the theorem, the assumption $D_\ell<1$ was imposed to conclude that $\dim_{\mathcal H} K_\ell=1-D_\ell$. In fact, this assumption can be weakened to the natural assumption~\eqref{eq-non-Shepp}. This fact may be well-known to experts; we include a proof in the Appendix \ref{appendix:Kahane-remark} for completeness.
\end{remark}

Besides the Hausdorff dimension, the Fourier dimension is another important notion that reflects the geometric properties of a fractal set from the viewpoint of harmonic analysis. For a Borel set $A\subset \T$, its Fourier dimension is defined by (see \cite[Section~8.2]{BSS23} and \cite[Chapter~17]{Kah85})
\begin{equation*}
\dim_{\mathcal F}(A):=\sup\Bigl\{0<\tau\leq 1:\ \exists \,\mu \in \mathcal M_{+}(A) \text{ such that } |\widehat{\mu}(n)|^2= O\big(n^{-\tau}\big) \Bigr\},
\end{equation*}
where $\mathcal M_{+}(A)$ denotes the set of non-zero finite  positive Borel measures on $\T$ that give full measure to $A$. In general, for any Borel set $A \subset \T$, it holds that 
\begin{align}\label{eq-F-H}
	\dim_{\mathcal F} A \leq \dim_{\mathcal H} A.
\end{align}
A Borel set is called {\it Salem} if its Fourier dimension and Hausdorff dimension coincide.  
For general theory and applications of these dimensions and Salem sets, we refer to \cite{BSS23,EkSch17,Kah85} and the references therein.  

Our main result is the following
\begin{theorem}\label{theorem-main}
	Assuming \eqref{eq-non-Shepp}, the uncovered set $K_\ell$ is almost surely Salem with Fourier dimension \[\dim_{\mathcal F} K_\ell = \dim_{\mathcal H} K_\ell = 1-D_\ell.\]
\end{theorem}

\subsection{Proof sketch and new ingredients in Dvoretzky random coverings}\label{s1.3}

To prove Theorem \ref{theorem-main}, it suffices to establish the almost sure lower bound $$\dim_{\mathcal F} K_\ell \geq 1 - D_\ell,$$ since the converse inequality already follows from Kahane's theorem and \eqref{eq-F-H}. Furthermore, without loss of generality, we may henceforth assume $D_\ell<1$, since any Borel set with Hausdorff dimension zero is trivially Salem.

Roughly speaking, we will deduce the above lower bound by examining the Fourier decay of $\mu_{\text{RC}}$,  {\it the multiplicative chaos measure associated with the Dvoretzky random covering}, which is a random measure supported on $K_\ell$ and is defined as the almost sure weak limit of a measure-valued martingale $\{\mu_k\}_{k=1}^\infty$(see \S \ref{sec-mc} below for its precise definition): 
\begin{align}\label{eq-meas-mart}
	\mu_k \xrightarrow[\text{a.s.}]{\text{weakly}} \mu_{\text{RC}}.
\end{align}
Specifically, the essential part of our argument is to prove
\begin{align}\label{eq-goal-1}
	\big|\widehat{\mu_{\text{RC}}}(n)\big|^2=O(n^{-\tau})~\text{ a.s. for any}~\tau<1-D_\ell.
\end{align}
The above Fourier decay of $\mu_\RC$ implies that,   {\it conditioned on the event $\{\mu_\RC \ne 0\}$}, the set $K_\ell$ is almost surely Salem.  For removing the conditioning on the event $\{\mu_\RC\ne 0\}$ in the above statement,  we shall need to introduced a sequence of modified multiplicative chaos measures; see Remark~\ref{rem-vanish} and Lemma~\ref{lem-zero-one-law} in   \S \ref{sec-mc} below for details on this subtle issue.

The proof of \eqref{eq-goal-1} constitutes the core of our argument, and it relies on the vector-valued martingale method, with some important new ingredients incorporated.

\subsubsection{A quick review of the vector-valued martingale method.}
The vector-valued martingale method, initially introduced by Chen--Han--Qiu--Wang \cite{CHQW24} in the setting of  Mandelbrot cascades, and later applied to Gaussian multiplicative chaos in \cite{CLQ25, LQT25I,LQT25II}, has proved to be particularly effective and transparent in establishing the sharp polynomial Fourier decay for various models of multiplicative chaos.    The reader is also refered to \cite{CLS24, GV23} for other recent works on Fourier decays of various models of multiplicative chaos measures.

A simple observation motivating this method is that: to prove the assertion \eqref{eq-goal-1}, it suffices to prove: 
\begin{align}\label{eq-goal-2}
	\forall\,\tau <1-D_\ell, \,\exists\,p,q \in (0,\infty), \, \text{such that} \,\,\E\bigg[\bigg\{\sum_{n=1}^\infty \big| n^{\tau/2} \widehat{\mu_{\text{RC}}}(n)\big|^{q}\bigg\}^{p/q}\bigg]<\infty.
\end{align}

Furthermore, the core idea for establishing \eqref{eq-goal-2} is to reduce it, via the next proposition, to studying the $L^p(\ell^q)$-norm of the vector-valued martingales associated with the measure-valued martingale in \eqref{eq-meas-mart}.

\begin{proposition}\label{prop-vector-valued-martingale}
	For any fixed $\tau<1-D_\ell$ and any $k\ge 1$, define the random sequences 
$$\mathbf{M}_k = \mathbf{M}_k^{(\tau)}:=\Big( n^{\tau /2 } \widehat{\mu_k}(n)\Big)_{n\geq 1} \quad \text{and} \quad  \mathbf{M} = \mathbf{M}^{(\tau)}:=\Big( n^{\tau/2} \widehat{\mu_{\text{RC}}}(n)\Big)_{n\geq 1}, $$
as well as the sequence of differences for $(\mathbf M_k)_{k\geq 1}$:
$$
\mathbf{D}_k = \mathbf{M}_k - \mathbf{M}_{k-1}, \quad \text{for $k\ge 2$ \, and\,\,}~\mathbf{D}_1= \mathbf{M}_1.
$$
Consider the following statements: 
\begin{enumerate}[label=(\roman*),font=\normalfont]
		\item \label{item:VVM-fourier-decay} The assertion \eqref{eq-goal-1} holds true. 
		\item \label{item:VVM-Lqlp-norm} For any $\tau <1-D_\ell$, there exist $1<p\leq 2\le q<\infty$, such that $$  \E\Big[\big\|\mathbf M\big\|^p_{\ell^q}\Big] <\infty.$$
		\item \label{item:VVM-Lqlp-uniform} For any $\tau <1-D_\ell$, there exist $1<p\leq 2\le q<\infty$, such that $$\sup_{k\geq 1}\E\Big[\big\|\mathbf M_k \big\|^p_{\ell^q}\Big] < \infty.$$
		\item \label{item:VVM-decoupling} For any $\tau <1-D_\ell$, there exist $1<p\leq 2\le q<\infty$, such that 
		\begin{align}\label{eq-goal-3}
			\sum_{k=1}^\infty  \E\Big[\big\|\mathbf D_k \big\|^p_{\ell^q}\Big]  <\infty.
		\end{align}
\end{enumerate}

	Then we have the implications $\textnormal{\ref{item:VVM-decoupling}}\Rightarrow \textnormal{\ref{item:VVM-Lqlp-uniform}}\Rightarrow \textnormal{\ref{item:VVM-Lqlp-norm}} \Rightarrow \textnormal{\ref{item:VVM-fourier-decay}}$.
\end{proposition}

The main tools used in the proof of Proposition \ref{prop-vector-valued-martingale} are the martingale convergence theorem (vector-valued version) and Pisier's martingale type-$p$ inequality. A detailed discussion of a general version of this proposition and its applications to Fourier decay of multiplicative chaos measures can be found in \cite{LQT25II}.  Note also that, for fixed $k$, the measure $\mu_k$ has a density taking constants on a finite number of random intervals and thus $|\widehat{\mu_k}(n)| = O(n^{-1})$, this  already gives the hint that, almost surely $\mathbf{M}_k \in \ell^q$ for large $q$ depending only on $\tau$ (but does not depend on $k$); see  Lemma~\ref{lem-pq} for the  precise choices of $p, q$. 

For the sake of completeness, we present a quick proof here.

\begin{proof}[Proof of Proposition \ref{prop-vector-valued-martingale}]
	The implication $\textnormal{\ref{item:VVM-Lqlp-norm}}\Rightarrow \textnormal{\ref{item:VVM-fourier-decay}}$ is trivial. We first prove $\textnormal{\ref{item:VVM-Lqlp-uniform}}\Rightarrow \textnormal{\ref{item:VVM-Lqlp-norm}}$. It is routine to verify that $(\mathbf M_k)_{k\geq 1}$ is an $\ell^q$-valued martingale for large enough $q$ which depends only on $\tau$. Since $p>1$, under the assumption of \ref{item:VVM-Lqlp-uniform}, the martingale convergence theorem (for $\ell^q$-valued martingales; see \cite[Chapter 2]{Pis16}) implies that $(\mathbf M_k)$ converges to some limit $\mathbf M'$, both in the $L^p(\ell^q)$-norm and almost surely. By the weak convergence \eqref{eq-meas-mart}, for each $n\geq 1$, we have $\widehat{\mu_k}(n) \to \widehat{\mu_{\text{RC}}}(n)$, and hence $$\mathbf M'(n)= \lim_{k\to\infty} \mathbf M_k(n)= \lim_{k\to \infty} n^{\tau/2} \widehat{\mu}_k(n) =n^{\tau/2} \widehat{\mu_{\text{RC}}}(n).$$ This shows $\mathbf M'=\mathbf M$, and the statement \ref{item:VVM-Lqlp-norm} follows.
	
	Next we prove $\textnormal{\ref{item:VVM-decoupling}}\Rightarrow \textnormal{\ref{item:VVM-Lqlp-uniform}}$. Since for any $2\le q<\infty$,  the Banach space $\ell^q$ is of martingale type $p$ for any $1<p\le 2$ (see \cite[Chapter 10]{Pis16}), by Pisier's type-$p$ inequality we have $$ \E\Big[\big\|\mathbf M_k \big\|^p_{\ell^q}\Big] \leq C_p\cdot  \sum_{j=1}^k \E\Big[\big\|\mathbf D_j \big\|^p_{\ell^q}\Big]$$ for a constant $C_p$ depending only on $p$. The proof is completed by taking the supremum over $k\ge 1$.
\end{proof}

\subsubsection{New ingredients}\label{sec-new-ingre}
By the discussions in the previous subsection, and in particular 
by Proposition~\ref{prop-vector-valued-martingale}, to prove Theorem~\ref{theorem-main}, one only needs to prove \eqref{eq-goal-3}. Therefore, we need an efficient upper estimate for
\begin{align}\label{goal-Dj}
\E\Big[\big\|\mathbf D_k \big\|^p_{\ell^q}\Big] = \E \Bigg[ \Bigg\{\sum_{n=1}^\infty n^{\frac{\tau q}{2}} \bigg|\int_0^1 D_k(t) e^{-2\pi i n t} dt\bigg|^q \Bigg\}^{p/q}\Bigg], 
\end{align}
where the precise definition of $D_k(t)$ is given in \eqref{def-Dk} below. Indeed, the estimate of \eqref{goal-Dj} still requires a more delicate analysis.

{\flushleft \bf A: Weak spatial independence in previous works:}
In previous studies on Fourier decays for Gaussian multiplicative chaos \cite{LQT25I}, for canonical Mandelbrot cascades \cite{CHQW24}, and for canonical Mandelbrot random coverings on the real line \cite{LQT25II}, the difficulty in estimating quantities similar to \eqref{goal-Dj} was overcome by leveraging a certain weak spatial independence of the generating process $\{X_k(t): t \in \mathcal T\}$ (see \cite[Assumption 1]{LQT25II} for a precise formulation, and see also \cite{SS18} for the original definition of spatially independent martingales). This weak spatial independence  allows for the reapplication of Pisier's martingale type-$p$ inequality, which further reduces \eqref{goal-Dj} to several localization estimates.

{\flushleft \bf B: Goes beyond weak spatial independence in current paper via translation-cancellation trick, small $L^1$-modulus of continuity and  small mean-oscillation of the sample  paths:}
In the current context of the Dvoretzky random covering model, the associated process $X_k$ (defined as in \eqref{eq-X_k} below) does not possess any weak spatial independence. To obtain the efficient upper estimate for \eqref{goal-Dj}, a pivotal step in our approach involves applying a variant of the {\it translation-cancellation trick} (see Lemma \ref{lemma-translation} below) to obtain a pointwise upper estimate of the supremum of a certain collection of high-frequencies of $D_k$. More precisely, we shall use the following simple $L^1$-modulus of continuity estimate: for small  $h>0$ and the subset of positive integers  $\Delta(h): = \big\{n \in\mathbb N_+: |e^{2\pi i n h}-1|\geq 1\big\}$, one has $$\sup_{n\in \Delta(h)} \bigg| \int_0^1 D_k(t) e^{-2\pi i n t}dt \bigg| \leq  \underbrace{\int_0^1 \big| D_k(t+h) - D_k(t)\big|dt}_{\text{denoted by $\omega^{L^1}_{D_k} (h)$}}.$$
Note that the quantity $\omega^{L^1}_{D_k} (h)$ in the above inequality is precisely the $L^1$-modulus of continuity of the sample path of $D_k$, which can be further controlled by the $L^1$-modulus of continuity of the sample paths of $X_j$ for $1\le j \le k$. 

Indeed, our subsequent proof requires an efficient upper estimate for the $p$-moment of $\omega^{L^1}_{D_k} (h)$: 
\[
\E \Big[\Big\{\omega^{L^1}_{D_k} (h)\Big\}^p\Big]= \E\bigg[ \Big(\int_0^1 \big| D_k(t+h) - D_k(t)\big|dt\Big)^p\bigg]. 
\]
The success of our approach also fundamentally relies on  small $L^1$-modulus of continuity and  small mean-oscillation of the sample paths of the generating process $X_k(t)$. More precisely, $D_k$ has the following form 
\[
X_k(t)  := \frac{1 - \indi_{I_k}(t)}{1- \ell_k}, \, M_k(t)  :=   \prod_{j=1}^{k} X_j(t) 
\] 
and 
\[
\text{$D_k(t): =M_{k-1}(t) \cdot (X_k(t)-1)$ for $k\ge 2$ and $D_1(t)= M_1(t)$. }
\]
Then by the triangle inequality, 
\[
| D_k(t+h) - D_k(t)| \le  | M_{k-1}(t+h)\cdot ( X_k(t+h)-X_k(t))|+ |(M_{k-1}(t+h)-M_{k-1}(t)) \cdot ( X_k(t)-1)|. 
\]  
Combining this with an application of Minkowski's inequality for the conditional expectation $\E^{(k)}[\,\cdot\,]: = \E[\,\cdot\, | \sigma(X_k)]$, we can deduce that (see the proof of Proposition~\ref{prop-p-moment-Dk} below for  details): 
\begin{align*}
\E^{(k)} \Big[\Big\{\omega^{L^1}_{D_k} (h)\Big\}^p\Big] \lesssim_p  &  \Bigg\{\int_0^1 \Big( \E\Big[ \big|M_{k-1}(t+h)\big|^p  \Big]\Big)^{1/p}\big|X_k(t+h)-X_k(t)\big|  dt \Bigg\}^{p}  
\\
& + \Bigg\{\int_0^1 \Big(\E\Big[ \big| M_{k-1}(t+h)-M_{k-1}(t) \big|^p\Big] \Big)^{1/p}\big|X_k(t)-1\big| dt \Bigg\}^{p}. 
\end{align*}
The next simple observation is that, by translation-invariance of the distributions of all our processes $X_k$, both of the following quantities are constants (independent of $t, h$) and they can be easily computed or upper-estimated (see Proposition \ref{prop-basic-facts} below): 
\[
A_k(p) : = \E\Big[ \big|M_{k-1}(t+h)\big|^p  \Big],   \quad  B_k(p) = \E\Big[ \big| M_{k-1}(t+h)-M_{k-1}(t) \big|^p\Big]. 
\]
Thus one obtains 
\begin{align*}
\E^{(k)} \Big[\Big\{\omega^{L^1}_{D_k} (h)\Big\}^p\Big] \lesssim_p  &  A_k(p) \cdot   \Bigg( \underbrace{\int_0^1 \big|X_k(t+h)-X_k(t)\big|  dt}_{\text{controlled by $L^1$-modulus of continuity of $X_k$}} \Bigg)^{p}   + B_k(p) \cdot \Bigg( \underbrace{\int_0^1\big|X_k(t)-1\big| dt}_{\text{small mean-oscillation}} \Bigg)^{p}. 
\end{align*}
From the above inequality, one immediately sees the roles played by the $L^1$-modulus of continuity of the sample path $X_k$ as well as  the  small mean-oscillation of the sample path $X_k$ around its mean value $\E[X_k(t)] = 1$.     See \S \ref{sec-proof} for the precise arguments.

\subsection{Further discussions and remarks}

Note that any Borel set with Hausdorff dimension zero is trivially Salem.  Particularly, in the critical case $D_\ell=1$, the uncovered set $K_\ell$ cannot support any non-zero positive measure with polynomial Fourier decay. This necessitates a further study of Fourier decay in the critical case, and the following problems arise naturally:

\begin{problem}[About the set $K_\ell$] 
	Assuming \eqref{eq-non-Shepp} and $D_\ell=1$, determine the additional conditions on the sequence $\ell$ under which $K_\ell$ almost surely supports a non-zero Rajchman measure (that is, a measure whose Fourier coefficients vanish at infinity). 
\end{problem}
A related question to ask is whether the multiplicative chaos measure $\mu_{\text{RC}}$ is almots surely Rajchman:
\begin{problem}[About the measure $\mu_{\text{RC}}$]
Assuming \eqref{eq-non-Shepp} and $D_\ell=1$, determine the necessary and sufficient conditions on the sequence $\ell$ under which the measure $\mu_{\text{RC}}$ is almost surely Rajchman.
\end{problem}

It should be noted that a partial answer to Problem~2 was provided by M. Tan in his master's thesis (see also his recent preprint \cite{Tan25}). Specifically, Tan proved that the condition 
\begin{equation}\label{eq-Tan-condition-Rajchman}
\sum_{k=1}^\infty (\ell_k-\ell_{k+1}) \exp\big( \ell_1+\cdots +\ell_k\big)<\infty
\end{equation}
is sufficient for $\mu_{\text{RC}}$ to be almost surely Rajchman.
In the following remark, we observe that the criterion \eqref{eq-Tan-condition-Rajchman} also provides a sufficient condition for $\mu_{\text{RC}}$ to be Rajchman in the $D_\ell=1$ case. 
\begin{remark}
To confirm that criterion \eqref{eq-Tan-condition-Rajchman} provides a non-vacuous sufficient condition for non-zero $\mu_{\text{RC}}$ to be Rajchman in the $D_\ell=1$ case, consider the sequence
\[
\ell_k=\frac{1}{k+10}\cdot \bigg(1-\frac{1}{\sqrt{\log(k+10)}}\bigg),\quad \forall k\ge 1.
\]
This sequence satisfies the necessary prerequisites: it meets \eqref{eq-non-Shepp} (producing a non-zero multiplicative measure $\mu_{\text{RC}}$) and satisfies $D_\ell=1$. Crucially, it also fulfills  \eqref{eq-Tan-condition-Rajchman}, thus confirming that the criterion is non-vacuous in this regime.
\end{remark}

\begin{remark}\label{remark:Tan-Dl<1-verification}
We also observe that for non-increasing sequences $\ell_k$, condition \eqref{eq-Tan-condition-Rajchman} is automatically satisfied when $D_\ell<1$ (see Appendix \ref{appendix:Tan-Dl<1-verification} for details).
\end{remark}

\vspace{0.3cm}

\noindent{\it Notation.} We denote by $\mathbb N_+ = \{1, 2, \dots\}$ the set of positive integers. By $A \lesssim B$ we mean there exists a constant $C>0$, depending only on the fixed parameters but independent of the main variables under consideration, such that $A \le C B$. Moreover, if $A\lesssim B$ and $B\lesssim A$, we write $A\asymp B$. By convention, the empty product $\prod_{j\in \emptyset} c_j$ for any numbers $c_j$ is defined to be $1$ and the empty sum $\sum_{j\in \emptyset} c_j$ is defined to be $0$.

\subsection*{Acknowledgements.} 
YQ is supported by National Natural Science Foundation of China (NSFC No. 12471145).

\section{The multiplicative chaos measures for Dvoretzky random coverings}\label{sec-mc}
The multiplicative chaos measure for Dvoretzky random covering model, which we denote by $\mu_{\text{RC}}$, was introduced by Kahane in his 1968 book (the first edition of \cite{Kah85}). Kahane pointed out that a good method to study the Dvoretzky random covering model is to define the following random products: 
\begin{align}\label{eq-X_k}
	M_k(t) := \prod_{j=1}^k X_j(t), \quad \text{with} \,\, X_j(t) := \frac{1 - \indi_{I_j}(t)}{1 - \ell_j}, \quad t \in \mathbb{T}, \, k \in \mathbb{N}_+,
	\end{align}
	where $I_j = (\omega_j, \omega_j + \ell_j)$ is the $j$-th random arc of length $\ell_j$, and then consider the random measures
	\begin{align*}
	d\mu_k(t):= M_{k}(t)dt, \quad k\in\mathbb N_+.
	\end{align*}
	Note that $(X_k )_k$ is a sequence of independent stochastic processes satisfying $X_k(t)\geq 0$, $\E[X_k(t)]\equiv 1$. Consequently, $(\mu_k)_{k\geq 1}$ forms a positive measure-valued martingale with respect to the natural filtration $(\mathscr{F}_k)_{k\ge 1}$ given by 
	\[
	\mathscr F_k =  \sigma(X_1,  X_2, \ldots, X_k) =  \sigma(\omega_1, \omega_2, \ldots, \omega_k)  .
	\] Kahane \cite{Kah87, Kah00} established the almost sure weak convergence of $(\mu_k)_{k\geq 1}$, and the weak limit is defined as the multiplicative chaos measure $\mu_{\text{RC}}$: 
\begin{align*}
	\mu_k \xrightarrow[\text{a.s.}]{\text{weakly}} \mu_{\text{RC}}.
\end{align*}
Note that  the  support of $\mu_\RC$ satisfies 
\[
\supp(\mu_{\text{RC}})  \subset \T \setminus \bigcup_{k = 1}^\infty I_k \subset K_\ell. 
\]
The random measure $\mu_{\text{RC}}$ is called {\it non-degenerate} if 
\[
\PP(\mu_{\text{RC}} \neq 0) >0. 
\]
Kahane~\cite[\S 3.1: Random coverings]{Kah87} proved that under the assumption \eqref{eq-non-Shepp}, the sequence $(\mu_k(\T))_{k\ge 1}$  of total masses is an $L^2$-bounded martingale and hence $\E[\mu_{\text{RC}}(\T)] = 1$.  Consequently, the assumption \eqref{eq-non-Shepp} implies that $\mu_{\text{RC}}$ is non-degenerate.  
\begin{remark}\label{rem-vanish}
It is easy to see that $\PP(\mu_\RC=0)>0$ (since $\sum_{k}\ell_k>1$). Therefore,
even under the assumption \eqref{eq-non-Shepp}, the random measure $\mu_\RC$ is non-degenerate but one always has
$
\PP(\mu_\RC =0)>0. 
$ 
\end{remark}

For our purpose of determining the Fourier dimension of the uncovered set $K_\ell$,  we shall need to examine some non-zero measure supported on $K_\ell$.  Therefore, in view of Remark \ref{rem-vanish}, we shall also need the seqeunce of multiplicative chaos measures $(\mu_{\text{RC}}^{\geqslant k})_{k=1}^\infty$ defined  by the following limit: 
\begin{align}\label{mu-RC-k}
\Big[\prod_{j=k}^m X_j(t) \Big]dt \xrightarrow[\text{a.s.}]{\text{weakly}} \mu_{\text{RC}}^{\geqslant k}. 
\end{align}
By definition, one has $\mu_{\text{RC}}^{\geqslant 1}  = \mu_{\text{RC}}$.  We emphasize that the sequence of  multiplicative chaos measures $(\mu_{\text{RC}}^{\geqslant k})_{k=1}^\infty$ are all coupled and satisfy the following relations: for $k, n \in \mathbb{N}_+$, 
\begin{align}\label{mu-k-rel}
\mu_{\text{RC}}^{\geqslant k} (dt)  = \Big[\prod_{j=k}^{k+n-1}X_j(t) \Big] \cdot \mu_{\text{RC}}^{\geqslant k+n}(dt),  \quad \text{a.s.} 
\end{align}
The above equation imply that $\big(\supp(\mu_\RC^{\geqslant k})\big)_{k=1}^\infty$ are nested subsets of $K_\ell$: 
\[
\supp(\mu_\RC^{\geqslant 1}) \subset \cdots \subset \supp(\mu_{\RC}^{\geqslant k}) \subset \supp(\mu_\RC^{\geqslant k+1}) \subset \cdots \subset K_\ell. 
\]

The following lemma will be useful for us. Its proof seems to be known to the experts and we include its simple proof for completeness. 
\begin{lemma}\label{lem-zero-one-law}
Assuming \eqref{eq-non-Shepp},  one has 
\[
\PP\Big( \big\{ \mu_\RC^{\geqslant k} = 0 \,\, \text{for all $k\ge 1$} \big\}\Big) = 0. 
\]
That is, under the assumption \eqref{eq-non-Shepp}, almost surely, there exists $k\ge 1$ such that the measure $\mu_\RC^{\geqslant k} \ne 0$. 
\end{lemma}
 \begin{proof}
Lemma \ref{lem-zero-one-law} is proved by using Kolmogorov's zero-one law. Indeed, the relation \eqref{mu-k-rel} implies that 
\[
\big\{ \mu_\RC^{\geqslant k+1} = 0 \big\}   \subset \big\{ \mu_\RC^{\geqslant k} = 0 \big\}, \quad k\in \mathbb{N}_+.
\]
Consequently,  for any $N\ge 1$, one has 
\[
\big\{ \mu_\RC^{\geqslant k} = 0 \,\, \text{for all $k\ge 1$} \big\} = \bigcap_{k\ge 1} \big\{ \mu_\RC^{\geqslant k} = 0 \big\} =   \bigcap_{k\ge N} \big\{ \mu_\RC^{\geqslant k} = 0 \big\}.
\]
It follows that 
\[
\big\{ \mu_\RC^{\geqslant k} = 0 \,\, \text{for all $k\ge 1$} \big\}   \in \bigcap_{N\ge 1} \sigma\big(X_N, X_{N+1},\cdots \big).
\]
Since $X_j$ are all independent, by Kolmogorov's zero-one law, one obtains 
\begin{align}\label{eq-01}
\PP\Big( \big\{ \mu_\RC^{\geqslant k} = 0 \,\, \text{for all $k\ge 1$} \big\}\Big) \in \{0,1\}. 
\end{align}
Now, since the assumption \eqref{eq-non-Shepp} implies that $\mu_\RC$ is non-degenerate (that is, $\PP(\mu_\RC \ne 0) >0$), one has 
\begin{align}\label{eq-less1}
\PP\Big( \big\{ \mu_\RC^{\geqslant k} = 0 \,\, \text{for all $k\ge 1$} \big\}\Big) \le \PP(\mu_\RC = 0) = 1 - \PP(\mu_\RC \ne 0)<1.
\end{align}
Combining  \eqref{eq-01} and \eqref{eq-less1}, we complete the whole proof of the lemma. 
 \end{proof}

   \section{Proof of Theorem~\ref{theorem-main} conditioned on $\{\mu_\RC\ne 0\}$}\label{sec-proof}

   This section is devoted to the proof of the following 
   \begin{proposition}\label{theorem-submain}
   Assuming \eqref{eq-non-Shepp}, then 
   \[
  \big|\widehat{\mu_{\RC}}(n)\big|^2=O(n^{-\tau})~\text{ a.s. for any}~\tau<1-D_\ell.
   \]
   In particular, conditioned on the event $\{\mu_\RC \ne 0\}$, the uncovered set $K_\ell$ is almost surely Salem with Fourier dimension \[\dim_{\mathcal F} K_\ell = \dim_{\mathcal H} K_\ell = 1-D_\ell.\]
   \end{proposition}

       As we discussed in \S \ref{s1.3} (see in particular Proposition \ref{prop-vector-valued-martingale}), to prove Proposition \ref{theorem-submain}, we only need to prove \eqref{eq-goal-3}. Namely, for any $\tau<1-D_\ell$, there exist $1<p\leq 2\leq q<\infty$ such that 
    \begin{align}\label{eq-goal-4}
    	\sum_{k=1}^\infty \E \Bigg[ \Bigg\{\sum_{n=1}^\infty n^{\frac{\tau q}{2}} \bigg|\int_0^1 D_k(t) e^{-2\pi i n t} dt\bigg|^q \Bigg\}^{p/q}\Bigg]<\infty,
    \end{align}
    where  $D_k$ are the martingale differences defined as 
    \begin{align}\label{def-Dk}
     D_k(t):= M_k(t)-M_{k-1}(t),\text{ for }k\geq 2, \text{ and }D_1(t):= M_1(t). 
     \end{align}
     
    \subsection{Pointwise upper estimate}
    We begin by providing {\it pointwise upper estimate} of the summands in \eqref{eq-goal-4}:
    \[
    \Bigg\{\sum_{n=1}^\infty n^{\frac{\tau q}{2}} \bigg|\int_0^1 D_k(t) e^{-2\pi i n t} dt\bigg|^q \Bigg\}^{p/q}, \quad k \ge 1. 
    \]

    \subsubsection{Dyadic frequency decomposition} 
		 For each integer $k\ge 1$, consider the dyadic blocks 
		 \begin{align}\label{def-Dkm}
		 \Delta_{k,m}:=\big\{n \in \mathbb N_+: 2^{m-1} \ell_k^{-1} \leq n < 2^{m} \ell_k^{-1} \big\},\quad m\geq 1,
		 \end{align} and $$\Delta_{k,0}:= \big\{n\in \mathbb N_+: 1\leq n< \ell_k^{-1}\big\}.$$ Since $0<p/q<1$, by the sub-additive inequality $(x+y)^{p/q}\le x^{p/q} + y^{p/q}$,  we have
    \begin{align*}
    	\Bigg\{\sum_{n=1}^\infty n^{\frac{\tau q}{2}} \bigg|\int_0^1 D_k(t) e^{-2\pi i n t} dt\bigg|^q \Bigg\}^{p/q}
   	= & \Bigg\{\sum_{m=0}^\infty \sum_{n\in \Delta_{k,m}} n^{\frac{\tau q}{2}} \bigg|\int_0^1 D_k(t) e^{-2\pi i n t} dt\bigg|^q \Bigg\}^{p/q}\\
    	\leq &\sum_{m=0}^\infty \Bigg\{\sum_{n\in \Delta_{k,m}} n^{\frac{\tau q}{2}} \bigg|\int_0^1 D_k(t) e^{-2\pi i n t} dt\bigg|^q \Bigg\}^{p/q}\\\
    	\leq & \sum_{m=0}^\infty \bigg(\sum_{n\in \Delta_{k,m}} n^{\frac{\tau q}{2}}\bigg)^{p/q} \cdot \sup_{n\in \Delta_{k,m}}\bigg|\int_0^1 D_k(t) e^{-2\pi i n t} dt\bigg|^p.
    \end{align*}

		Since $\Big(\sum_{n\in \Delta_{k,m}} n^{\frac{\tau q}{2}}\Big)^{p/q}\asymp \Big(2^m \ell_k^{-1}\Big)^{\frac{\tau p}{2} + \frac{p}{q}}$, we deduce that
		\begin{align}\label{eq-dyadic-frequency}
			\Bigg\{\sum_{n=1}^\infty n^{\frac{\tau q}{2}} \bigg|\int_0^1 D_k(t) e^{-2\pi i n t} dt\bigg|^q \Bigg\}^{p/q}\lesssim \sum_{m=0}^\infty \Big(2^m \ell_k^{-1}\Big)^{\frac{\tau p}{2} + \frac{p}{q}} \cdot \sup_{n\in \Delta_{k,m}}\bigg|\int_0^1 D_k(t) e^{-2\pi i n t} dt\bigg|^p.
		\end{align}

    For the low-frequency terms $n\in \Delta_{k,0}$, we could apply the crude bound
    \begin{align}\label{eq-crude-bound}
    \bigg|\int_0^1 D_k(t) e^{-2\pi i n t} dt\bigg|\le \int_0^1 |D_k(t)| dt,\quad \forall n\in \Delta_{k,0}.
		\end{align}
    This is sufficient for our further analysis. However, this bound fails to capture the oscillatory nature of the Fourier basis $e^{-2\pi i nt}$ for $n>\ell_k^{-1}$. To handle the high-frequency terms, we shall apply a variant of the translation-cancellation trick presented below.
    
    \subsubsection{A variant of the translation-cancellation trick for high-frequency terms}
    
    \begin{lemma}\label{lemma-translation}
    	Suppose $h>0$ and let $\Delta(h): = \big\{n \in\mathbb N_+: |e^{2\pi i n h}-1|\geq 1\big\}.$ Then $$\sup_{n\in \Delta(h)} \bigg| \int_0^1 D_k(t) e^{-2\pi i n t}dt \bigg| \leq \int_0^1 \big| D_k(t+h) - D_k(t)\big|dt.$$
    \end{lemma}

    \begin{proof}
    	By the translation invariance of Haar measure, we have 
    	\[
    	\int_0^1 D_k(t+h)e^{-2\pi i nt} dt=e^{2\pi inh}\int_0^1 D_k(t+h) e^{-2\pi i n(t+h)} dt=e^{2\pi i nh}\int_0^1 D_k(t)e^{-2\pi i nt} dt.
    	\]
    	Therefore
    	\begin{align*}
    		\int_0^1 \big[D_k(t+h)-D_k(t)\big] e^{-2\pi i n t} dt= \big(e^{2\pi i n h }-1\big) \int_{0}^1 D_k(t) e^{-2\pi i n t} dt,\quad \forall h>0.
    	\end{align*}
    	Since $|e^{2\pi i nh}-1|\ge 1$ for any $n\in \Delta(h)$, it follows that
    	\begin{align*} 
    		\bigg| \int_{0}^1 D_k(t) e^{-2\pi i n t} dt\bigg|  =\,& \big|e^{2\pi i n h}-1\big|^{-1}\cdot  \bigg| \int_0^1 \big[D_k(t+h)-D_k(t)\big] e^{-2\pi i n t} dt\bigg|\\
    		\leq &\int_0^1 \big| D_k(t+h) - D_k(t)\big|dt.
    	\end{align*}
			The proof is completed.
    \end{proof}

		Next, we choose $h$ appropriately to estimate the high-frequency Fourier coefficients. For fixed $k\geq 1$ and $m\geq 1$, we set $$h_{k,m}=2^{-m-2}\ell_k.$$ Recall the definition \eqref{def-Dkm} of $\Delta_{k,m}$.  Note that for any $n\in \Delta_{k,m},$ one has   $
     \pi/2\le 2\pi n h_{k,m}< \pi$ and hence 
     \begin{align}\label{eq-far-1}
     |e^{2\pi i nh_{k,m}}-1|\ge \sqrt{2}>1, \quad n\in \Delta_{k,m}
     \end{align}
     Comparing \eqref{eq-far-1} with the definition of $\Delta(h)$ in Lemma \ref{lemma-translation}, one has 
     \[
     \Delta_{k,m}\subset \Delta(h_{k,m}).
     \]
      Therefore, it follows from Lemma \ref{lemma-translation} that 
     \begin{align}\label{eq-high-frequency}
     \sup_{n\in \Delta_{k,m}}\bigg|\int_0^1 D_k(t) e^{-2\pi i n t} dt\bigg|\le \int_0^1 \big| D_k(t+h_{k,m}) - D_k(t)\big|dt,\quad m\geq 1.
	\end{align}
     
     Applying \eqref{eq-crude-bound} to bound the low-frequency terms ($m=0$) and \eqref{eq-high-frequency} to handle the high-frequency terms ($m\geq 1$), we derive from \eqref{eq-dyadic-frequency} the final pointwise estimate
     \begin{align*}
     \Bigg\{\sum_{n=1}^\infty n^{\frac{\tau q}{2}} \bigg|\int_0^1 D_k(t) e^{-2\pi i n t} dt\bigg|^q \Bigg\}^{p/q}
	 \lesssim \Big(\ell_k^{-1}\Big)^{\frac{\tau p}{2} + \frac{p}{q}}\cdot \bigg( \int_0^1 |D_k(t)|dt \bigg)^p&\\
	 +\sum_{m=1}^\infty \Big(2^m \ell_k^{-1}\Big)^{\frac{\tau p}{2} + \frac{p}{q}} \cdot \bigg(\int_0^1 \big| D_k(t+h_{k,m}) - D_k(t)\big|dt\bigg)^p&.
	\end{align*}
    Taking expectation on both sides, we obtain
     \begin{align}\label{eq:E-pw-estimate}
     \begin{split}
     	 \E\Bigg[\bigg\{\sum_{n=1}^\infty n^{\frac{\tau q}{2}} \bigg|\int_0^1  D_k(t) e^{-2\pi i n t} dt\bigg|^q  \bigg\}^{p/q} \Bigg]
		 \lesssim \Big(\ell_k^{-1}\Big)^{\frac{\tau p}{2} + \frac{p}{q}}\cdot \E\bigg[\bigg( \int_0^1 |D_k(t)|dt \bigg)^p\bigg]  & 
		 \\
		  +\sum_{m=1}^\infty \Big(2^m \ell_k^{-1}\Big)^{\frac{\tau p}{2} + \frac{p}{q}} \cdot \E\bigg[\bigg(\int_0^1 \big| D_k(t+h_{k,m}) - D_k(t)\big|dt\bigg)^p\bigg]&.
		 \end{split}
     \end{align}
     Therefore, as we mentioned in \S \ref{s1.3}, our next step is to provide effective estimates for the $p$-moments of the following random variables $$\int_0^1 |D_k(t)|dt \, \text{ and } \, \int_0^1 | D_k(t+h_{k,m}) - D_k(t)|dt.$$ 
	\subsection{Moment estimates for the $L^1$-norm and $L^1$-modulus of continuity of $D_k$.}

	Recall that
	\[
	D_k(t)=M_{k}(t)-M_{k-1}(t)=M_{k-1}(t)(X_k(t)-1).
	\] 
	This subsection is devoted to the proof of the following proposition.
	\begin{proposition}\label{prop-p-moment-Dk} The following statements hold true: 
		\begin{enumerate}[label=(\roman*),font=\normalfont]
		\item \label{item:prop-p-moment-difference} Uniformly for $k\geq 1$ and $h>0$, we have
		\[
      \E\bigg[\bigg( \int_0^1 |D_k(t+h)-D_k(t)|dt \bigg)^p\bigg]\lesssim \Big(h^p+hk\ell_k^p\Big)\cdot \prod_{j=1}^{k-1}(1-\ell_j)^{1-p}.
		\]
		\item \label{item:prop-p-moment-Dk}Uniformly for $k\geq 1$, we have $$\E\bigg[\bigg( \int_0^1 |D_k(t)|dt \bigg)^p\bigg]\lesssim  \ell_k^p\cdot \prod_{j=1}^{k-1}(1-\ell_j)^{1-p}.$$
		\end{enumerate}
	\end{proposition}

	We collect in the following proposition several basic facts necessary for the proof of Proposition~\ref{theorem-submain}. It should be emphasized that, albeit elementary, these facts are fundamental to the Fourier decay of the multiplicative chaos measure $\mu_{\text{RC}}$. For the reader's convenience, we recall that $$X_k(t)=\frac{1-\indi_{I_k}(t)}{1-\ell_k}, \quad M_k(t)=\prod_{j=1}^k X_j(t).$$

	\begin{proposition}\label{prop-basic-facts}
		 For any $1<p\le 2$, and any $k\geq 1$, the following statements hold:
		
		\begin{enumerate}[label=(\roman*),font=\normalfont]
			\item \label{item:basic-facts-p-moment-Xk} 
			For any $t\in\T$, we have $$\E\big[|X_k(t)|^p\big]=(1-\ell_k)^{1-p}.$$
			\item \label{item:basic-facts-L1-Xk-1}
			For an arbitrary (fixed) sample path, we have $$\int_0^1 |X_k(t)-1|dt =2\ell_k.$$
			\item \label{item:basic-facts-p-moment-Mk} 
			For any $t\in\T$, we have $$\E\big[|M_k(t)|^p\big]=\prod_{j=1}^k (1-\ell_j)^{1-p}.$$
			\item \label{item:basic-facts-L1-difference} 
			For an arbitrary (fixed) sample path, we have $$\int_0^1 \Big|X_k(t+h)-X_k(t)\Big| dt \leq \frac{2h}{1-\ell_k}, \quad \forall h>0.$$
			\item \label{item:basic-facts-p-moment-difference} 
			For any $t\in \T$ and $h>0$, we have $$\E\bigg[\Big|M_{k}(t+h)-M_k(t)\Big|^p\bigg]
				\leq\frac{2hk}{1-\ell_1}\prod_{j=1}^k (1-\ell_j)^{1-p}.$$
		\end{enumerate}
	\end{proposition}

	\begin{proof}
	The identities in \ref{item:basic-facts-p-moment-Xk}, \ref{item:basic-facts-L1-Xk-1} and \ref{item:basic-facts-p-moment-Mk} can be easily derived from the definition and the independence of $(X_k)_{k\geq 1}$. We now turn to the proof of \ref{item:basic-facts-L1-difference}. By definition of $X_k$, we have 
	\begin{align*}
		\int_0^1 \Big|X_k(t+h)-X_k(t)\Big| dt 
		=\frac{1}{1-\ell_k}\cdot \int_0^1  \Big|\indi_{I_k}(t+h)-\indi_{I_k}(t)\Big| dt.
	\end{align*}
	Note that the last integral equals the Lebesgue measure of the symmetric difference $I_k\Delta (I_k-h)$, which is bounded above by $2h$: 
	\begin{align}\label{set-diff}
	\int_0^1  \Big|\indi_{I_k}(t+h)-\indi_{I_k}(t)\Big| dt=\int_0^1 \Big|\indi_{I_k-h}(t)-\indi_{I_k}(t)\Big|dt=\int_0^1 \indi_{I_k\Delta (I_k-h)}(t)dt \leq 2h.
	\end{align}
	 This proves \ref{item:basic-facts-L1-difference}.  
	
	Now we turn to prove \ref{item:basic-facts-p-moment-difference} . Observe that the modulus of the difference between two indicator functions takes only the values $0$ and $1$, hence
	\begin{align*}
		\Big|M_{k}(t+h)-M_k(t)\Big|^p \cdot \prod_{j=1}^k (1-\ell_j)^{p} 
		&=\Big| \prod_{j=1}^k \indi_{I_j^c}(t+h) - \prod_{j=1}^k \indi_{I_j^c}(t) \Big|^p
		\\
		\text{(by the $0$-$1$ valued feature)} &		=\Big| \prod_{j=1}^k \indi_{I_j^c}(t+h) - \prod_{j=1}^k \indi_{I_j^c}(t) \Big|.
	\end{align*}
	Combining this with the elementary identity $$\prod_{j=1}^k a_j - \prod_{j=1}^k b_j= \sum_{r=1}^k  \Big(\prod_{1\leq j < r}a_j \Big) \Big(a_r-b_r\Big)  \Big(\prod_{r<j\leq k} b_j \Big),$$ we deduce that
	\begin{align*}
		\Big|M_{k}(t+h)-M_k(t)\Big|^p \cdot \prod_{j=1}^k (1-\ell_j)^{p}\leq &\sum_{r=1}^k \Big( \prod_{1\leq j< r} \indi_{I_j^c}(t+h) \Big) \Big|\indi_{I_r^c}(t+h)-\indi_{I_r^c}(t) \Big| \Big( \prod_{r< j\leq k} \indi_{I_j^c}(t)\Big).
	\end{align*}
	Taking expectations, and using the independence of $I_j$,
	we have 
	\begin{align*} 
		&\E\bigg[\Big|M_{k}(t+h)-M_k(t)\Big|^p\bigg] \cdot \prod_{j=1}^k (1-\ell_j)^{p}
		\\
		\leq & \sum_{r=1}^k \Big( \prod_{1\leq j < r} \E\big[ \indi_{I_j^c}(t+h) \big]  \Big) \cdot \E \Big[  \big|\indi_{I_r^c}(t+h)-\indi_{I_r^c}(t) \big|\Big] \cdot \Big( \prod_{r< j\leq k}  \E \big[\indi_{I_j^c}(t)\big] \Big) .
	\end{align*}

We apply the identity $|\indi_I(u)-\indi_{I'}(u)|=\indi_{I\Delta I'}(u)$ and the geometric bound used in \eqref{set-diff} again to obtain that for any $t\in \T$, 
\begin{align*}
\E\bigg[ \Big|\indi_{I_r^c}(t+h)-\indi_{I_r^c}(t)\Big| \bigg] &= \E\bigg[\Big|\indi_{I_r}(t+h)-\indi_{I_r}(t)\Big|\bigg]\\
&=\int_0^1 \Big|\indi_{(u,u+\ell_r)}(t+h)-\indi_{(u,u+\ell_r)}(t)\Big|du\\
&=\int_0^1 \Big|\indi_{(t+h-\ell_r,t+h)}(u)-\indi_{(t-\ell_r,t)}(u)\Big|du\le 2h.
\end{align*}
Combining this with 
\[
\E\big[ \indi_{I_j^c}(t+h) \big] = \E\big[ \indi_{I_j^c}(t) \big] = 1 - \ell_j, \quad \forall t \in \T \an h>0, 
\]
we conclude that
\begin{align*} 
		&\E\bigg[\Big|M_{k}(t+h)-M_k(t)\Big|^p\bigg] \cdot \prod_{j=1}^k (1-\ell_j)^{p}\leq 2h \sum_{r=1}^k \prod_{j\neq r} (1-\ell_j) \leq 2hk \prod_{j=2}^k (1-\ell_j),
\end{align*}	
where we used the fact 
	\[
	\prod_{j\neq r} (1-\ell_j) \le \prod_{j =2 }^k (1-\ell_j), \quad \forall 1\le r \le k. 
	\]
	This is equivalent to the desired inequality in \ref{item:basic-facts-p-moment-difference}.
	\end{proof} 

	Now we are ready to establish the moment estimates in Proposition~\ref{prop-p-moment-Dk}. 

	\begin{proof}[Proof of Proposition~\ref{prop-p-moment-Dk}] We first prove Proposition~\ref{prop-p-moment-Dk}-\ref{item:prop-p-moment-difference}. We observe the identity:
		\begin{align*} 
		D_k(t+h)-D_k(t)=&M_{k-1}(t+h)\cdot \big(X_k(t+h)-1\big)-M_{k-1}(t)\cdot \big(X_k(t)-1\big)\\
		=&M_{k-1}(t+h)\cdot \big( X_k(t+h)-X_k(t)\big)\\
		&+\big(M_{k-1}(t+h)-M_{k-1}(t)\big) \cdot \big( X_k(t)-1\big),
    \end{align*}
  which gives the bound
    \begin{align*}
    \int_0^1 |D_k(t+h)-D_k(t)|dt\leq &\int_0^1 |M_{k-1}(t+h)|\cdot |X_k(t+h)-X_k(t)|dt\\
    &+\int_0^1 |M_{k-1}(t+h)-M_{k-1}(t)|\cdot | X_k(t)-1|dt.
    \end{align*} 
Hence 
    \begin{align*}
    	\E\bigg[\bigg(\int_0^1 \big| D_k(t+h) - D_k(t)\big|dt\bigg)^p\bigg]&\lesssim\E\bigg[\bigg(\int_0^1 |M_{k-1}(t+h)|\cdot |X_k(t+h)-X_k(t)|dt\bigg)^p\bigg]\\
    	&\quad +\E\bigg[\bigg(\int_0^1 |M_{k-1}(t+h)-M_{k-1}(t)|\cdot | X_k(t)-1|dt\bigg)^p\bigg]\\
			& =:T_1 +T_2. 
    \end{align*}

		Now we proceed to the estimate for $T_1$. For simplicity, we denote by $\E^{(k)}$ the conditional expectation with respect to the $\sigma$-algebra  $\sigma(X_k) = \sigma(\omega_k)$ generated by the single random variable $X_k$ (one should not confuse $\E^{(k)}[\,\cdot\,]$ with $\E[\,\cdot\,| \sigma(X_1,\dots, X_k)]$): 
		\[
		\E^{(k)}\big[\cdot\big] =  \E\big[\,\cdot\, |  \sigma(X_k)\big]. 
		\]  
		 By the Minkowski inequality for the conditional expectation $\E^{(k)}[\,\cdot\,]$, we have 
		\begin{alignat}{2}
    &\quad&&\E^{(k)}\bigg[ \bigg( \int_0^1 M_{k-1}(t+h)\cdot \big|X_k(t+h)-X_k(t)\big|  dt \bigg)^p\bigg]\nonumber \\
    &\leq&&\Bigg\{\int_0^1 \Big(\E^{(k)}\Big[ \big|M_{k-1}(t+h)\big|^p  \Big]\Big)^{1/p}\big|X_k(t+h)-X_k(t)\big|  dt \Bigg\}^{p}.\label{eq-conditional-on-k-1}
		\end{alignat}
		Since $M_{k-1}$ is independent with $X_k$, and by Proposition~\ref{prop-basic-facts}-\ref{item:basic-facts-p-moment-Mk}, we have $$\E^{(k)}\Big[ \big|M_{k-1}(t+h)\big|^p  \Big]=\E\Big[ \big|M_{k-1}(t+h)\big|^p  \Big]=\prod_{j=1}^{k-1} (1-\ell_j)^{1-p},\quad t\in \T, \,\,h>0.$$ Substituting this into \eqref{eq-conditional-on-k-1}, and then applying Proposition~\ref{prop-basic-facts}-\ref{item:basic-facts-L1-difference}, we see 
		\begin{align*}
		\E^{(k)}\bigg[ \bigg( \int_0^1 M_{k-1}(t+h)\cdot \big|X_k(t+h)-X_k(t)\big| dt \bigg)^p\bigg] & \leq \bigg( \frac{2h}{1-\ell_k}\bigg)^p\cdot \prod_{j=1}^{k-1} (1-\ell_j)^{1-p}
	\\
	& \lesssim h^p \prod_{j=1}^{k-1} (1-\ell_j)^{1-p}.\end{align*}
			 By taking expectations we obtain $$T_1=\E\bigg[ \bigg( \int_0^1 M_{k-1}(t+h)\cdot \big|X_k(t+h)-X_k(t)\big| dt \bigg)^p\bigg]
			\lesssim h^p \prod_{j=1}^{k-1} (1-\ell_j)^{1-p}.$$

		Now we proceed to the estimate for $T_2$. Similarly, again by the Minkowski inequality for the conditional expectation $\E^{(k)}[\,\cdot\,]= \E[\,\cdot\, | \sigma(X_k)]$, we have
		\begin{align} 
			&\E^{(k)}\bigg[ \bigg( \int_0^1 \big| M_{k-1}(t+h)-M_{k-1}(t) \big| \cdot \big|X_k(t)-1 \big|dt \bigg)^p \bigg]\nonumber \\
			\leq &\Bigg\{\int_0^1 \Big(\E^{(k)}\Big[ \big| M_{k-1}(t+h)-M_{k-1}(t) \big|^p\Big] \Big)^{1/p}\big|X_k(t)-1\big| dt \Bigg\}^{p}. \label{eq-conditional-on-k-2}
		\end{align}
		By the independence of $M_{k-1}$ and $X_k$, and Proposition~\ref{prop-basic-facts}-\ref{item:basic-facts-p-moment-difference}, we have $$\E^{(k)}\Big[ \big| M_{k-1}(t+h)-M_{k-1}(t) \big|^p\Big] = \E\Big[ \big| M_{k-1}(t+h)-M_{k-1}(t) \big|^p\Big] \leq \frac{2hk}{1-\ell_1} \prod_{j=1}^k (1-\ell_j)^{1-p}.$$
		Substituting this into \eqref{eq-conditional-on-k-2} and then applying Proposition~\ref{prop-basic-facts}-\ref{item:basic-facts-L1-Xk-1}, we see $$T_2=\E^{(k)}\bigg[ \bigg( \int_0^1 \big| M_{k-1}(t+h)-M_{k-1}(t) \big| \cdot \big|X_k(t)-1 \big|dt \bigg)^p \bigg] \lesssim h k \ell_k^p \cdot \prod_{j=1}^{k-1} (1-\ell_j)^{1-p}.$$

		Finally, we conclude from the estimates for $T_1$ and $T_2$ that $$\E\bigg[\bigg(\int_0^1 \big| D_k(t+h) - D_k(t)\big|dt\bigg)^p\bigg] \leq T_1 + T_2 \lesssim \Big(h^p+h k \ell_k^p \Big)\cdot \prod_{j=1}^{k-1} (1-\ell_j)^{1-p}.$$ This proves Proposition~\ref{prop-p-moment-Dk}-\ref{item:prop-p-moment-difference}. 
		
		The proof of Proposition~\ref{prop-p-moment-Dk}-\ref{item:prop-p-moment-Dk} is analogous to those estimates for $T_1$ and $T_2$, where the Minkowski inequality for conditional expectations, Proposition~\ref{prop-basic-facts}-\ref{item:basic-facts-L1-Xk-1} and Proposition~\ref{prop-basic-facts}-\ref{item:basic-facts-p-moment-Mk} are applied. We spare the reader the details. The proof of Proposition~\ref{prop-p-moment-Dk} is completed.
	\end{proof}

    \subsection{Conclude Proposition~\ref{theorem-submain}}
    For the clarity of the proof, we state separately the choice of the exponents. 
    \begin{lemma}\label{lem-pq}
    If $0\le D_\ell<1$ and $0\le \tau <1 - D_\ell$, then there exists $1<p\le 2\le q<\infty$ such that 
    \begin{align}\label{choose-pq}
    1-\frac{\tau p}{2}-\frac{p}{q}>0 \an -1+\tau \cdot \big(1-\frac{p}{2}\big)-\bigg( \frac{1-\tau -D_\ell}{2}\bigg)(p-1)+ \frac{p}{q}<-1. 
    \end{align}
    \end{lemma}
    \begin{proof}
		Since $1- \tau >0$ and $1-\tau- D_\ell>0$, the above inequalities are satisfied by  taking $p=2$ and a sufficiently large $q$. 
    \end{proof}
    
	\begin{proof}[Proof of Proposition \ref{theorem-submain}] Assume that $D_\ell <1$.	Recall that, to conclude the proof of Proposition \ref{theorem-submain}, it suffices to prove that  for any $\tau<1-D_\ell$, there exist $1<p\leq 2\leq q<\infty$ such that 
    \begin{align*}
    	\sum_{k=1}^\infty \E \Bigg[ \Bigg\{\sum_{n=1}^\infty n^{\frac{\tau q}{2}} \bigg|\int_0^1 D_k(t) e^{-2\pi i n t} dt\bigg|^q \Bigg\}^{p/q}\Bigg]<\infty. 
    \end{align*}
    
    Now fix any $\tau$ with $0\le \tau < 1- D_\ell$.  By Lemma \ref{lem-pq}, we may choose exponents $p, q$ with $1<p \le 2 \le q <\infty$ satisfying the inequalities \eqref{choose-pq}. 
    Substituting estimates in Proposition~\ref{prop-p-moment-Dk} into \eqref{eq:E-pw-estimate}, we obtain 
     \begin{align} 
			&\E \Bigg[ \Bigg\{\sum_{n=1}^\infty n^{\frac{\tau q}{2}} \bigg|\int_0^1 D_k(t) e^{-2\pi i n t} dt\bigg|^q \Bigg\}^{p/q}\Bigg]\nonumber \\
			\lesssim & \Big(\ell_k^{-1}\Big)^{\frac{\tau p}{2}+\frac{p}{q}}\cdot \ell_k^p \prod_{j=1}^{k-1} (1-\ell_j)^{1-p}+ 
			\sum_{m=1}^\infty \Big(2^m \ell_k^{-1}\Big)^{\frac{\tau p}{2} + \frac{p}{q}} \Big( 2^{-mp} + 2^{-m} k \ell_k \Big)\cdot \ell_k^p \cdot \prod_{j=1}^{k-1} (1-\ell_j)^{1-p} \nonumber \\
			=& \Big(\ell_k^{-1}\Big)^{\frac{\tau p}{2}+\frac{p}{q}}\cdot \ell_k^p \cdot \prod_{j=1}^{k-1} (1-\ell_j)^{1-p} \cdot \bigg[ 1+ \sum_{m=1}^\infty  \Big(2^m \Big)^{\frac{\tau p}{2}+\frac{p}{q}} \Big(2^{-mp} + 2^{-m} k \ell_k \Big) \bigg]\nonumber \\
			=& \ell_k^{\,p-\frac{\tau p}{2}-\frac{p}{q}} \cdot \prod_{j=1}^{k-1} (1-\ell_j)^{1-p} \cdot \bigg[ 1+ \sum_{m=1}^\infty  \Big(2^{-m(p-\frac{\tau p}{2}-\frac{p}{q})} +k\ell_k 2^{-m (1-\frac{\tau p}{2}-\frac{p}{q})} \Big) \bigg].\label{eq-series-p-q-1}
		\end{align}
		We now incorporate the assumptions on the sequence $(\ell_k)_{k\geq 1}$. Since $\ell$ is non-increasing and  $$D_\ell=\limsup_{k\to \infty} \frac{\ell_1 + \cdots +\ell_k}{\log k}<1,$$ we have 
		\begin{align}\label{klk-log-bdd}
		k\ell_k\leq \sum_{j=1}^k \ell_j =O(\log k).
		\end{align}
		Moreover, it follows from \eqref{choose-pq} that 
	    \[
		p - \frac{\tau p}{2} - \frac{p}{q}>  1- \frac{\tau p}{2} - \frac{p}{q}>  0.
		\] 
		As a consequence, 
		\begin{align}\label{sum-to-log}
		\sum_{m=1}^\infty \Big(2^{-m(p-\frac{\tau p}{2}-\frac{p}{q})} +k\ell_k 2^{-m (1-\frac{\tau p}{2}-\frac{p}{q})} \Big) \lesssim 1+\log k.
		\end{align}
		 On the other hand,  take $C =  \sup_j (1 - \ell_j)^{-1} = (1-\ell_1)^{-1}>0$,  then  for all $j\ge 1$, 
\[
\frac{1}{1-\ell_j} =  1 + \ell_j + \frac{\ell_j^2}{1-\ell_j} \le  \exp \Big( \ell_j + \frac{ \ell_j^2}{1-\ell_j}\Big) \le   \exp \Big( \ell_j + C  \ell_j^2 \Big). 
\]
	Therefore, since $1<p\le 2$, one has 
\[
	 \prod_{j=1}^{k-1} (1-\ell_j)^{1-p} = \Big(\prod_{j=1}^{k-1} \frac{1}{1-\ell_j}\Big)^{p-1} \le \exp \Big ((p-1) \sum_{j=1}^{k-1}\ell_j + (p-1)C\sum_{j=1}^{k-1}\ell_j^2\Big).
\]
In view of \eqref{klk-log-bdd}, one has $\sum_{j=1}^\infty \ell_j^2 \lesssim \sum_{j=1}^\infty (\log j)^2 /j^2<\infty$ and hence, uniformly on $k$, 
\begin{align}\label{correct-upp-bdd}
 \prod_{j=1}^{k-1} (1-\ell_j)^{1-p} = \Big(\prod_{j=1}^{k-1} \frac{1}{1-\ell_j}\Big)^{p-1} \lesssim \exp \Big ((p-1) \sum_{j=1}^{k}\ell_j\Big).
 \end{align}
		Substituting  \eqref{sum-to-log} and \eqref{correct-upp-bdd} into \eqref{eq-series-p-q-1} and then summing over all $k\geq 1$, we arrive at 
		\begin{align}\label{eq-series-p-q-2}
			\sum_{k=1}^\infty \E \Bigg[ \Bigg\{\sum_{n=1}^\infty n^{\frac{\tau q}{2}} \bigg|\int_0^1 D_k(t) e^{-2\pi i n t} dt\bigg|^q \Bigg\}^{p/q}\Bigg] \lesssim \sum_{k=1}^\infty \ell_k^{\,p-\frac{\tau p}{2}-\frac{p}{q}} \cdot \exp\Big( \big(p-1\big)  \sum_{j=1}^k \ell_j\Big) \cdot \big( 1+\log k\big).
	\end{align} 
		
 Our assumption $0\le \tau <1- D_\ell$ implies that $D_\ell < 1 - \tau$ and hence 
 \[
 \limsup_{k\to \infty} \frac{\ell_1 + \cdots +\ell_k}{\log k} = D_\ell <\frac{1 - \tau + D_\ell}{2}. 
 \]
  Therefore, there exists $K\in \mathbb N_+$, such that $$\sum_{j=1}^k \ell_j \leq \bigg( \frac{1-\tau + D_\ell}{2}\bigg)\log k, \quad \forall k\geq K.$$ This implies the following estimates for the right hand side of \eqref{eq-series-p-q-2}:
	\begin{align*}
		&\sum_{k=K}^\infty \ell_k^{\,p-\frac{\tau p}{2}- \frac{p}{q}} \cdot \exp\Big( \big(p-1\big) \cdot \sum_{j=1}^k \ell_j\Big) \cdot \Big( 1+\log k\Big)\\
		\leq& \sum_{k=K}^\infty \ell_k^{\,p-\frac{\tau p}{2}- \frac{p}{q}}\cdot k^{(p-1)\frac{1-\tau+D_\ell}{2}}\cdot \Big( 1+\log k\Big)\\
		=&\sum_{k=K}^\infty k^{-1+\tau(1-\frac{p}{2})-\big( \frac{1-\tau -D_\ell}{2}\big)(p-1)+ \frac{p}{q}} \cdot \Big( k\ell_k\Big)^{p-\frac{\tau p}{2}- \frac{p}{q}}\cdot \Big( 1+\log k\Big)\\
		\lesssim &\sum_{k=K}^\infty k^{-1+\tau(1-\frac{p}{2})-\big( \frac{1-\tau -D_\ell}{2}\big)(p-1)+ \frac{p}{q}} \cdot \Big( 1+\log k\Big)^{p-\frac{\tau p}{2}- \frac{p}{q}+1 }.
	\end{align*} 
Since the exponents $p,q$ are chosen to satisfy the condition in Lemma \ref{lem-pq}, we have 
$$-1+\tau(1-\frac{p}{2})-\bigg( \frac{1-\tau -D_\ell}{2}\bigg)(p-1)+ \frac{p}{q}<-1.$$ 
This implies that   the series in \eqref{eq-series-p-q-2} converges and we complete the whole proof of Proposition~\ref{theorem-submain}.
	\end{proof}

\section{Proof of Theorem~\ref{theorem-main}}

Recall the definition of the multiplicative chaos measure $\mu_\RC^{\geqslant k}$ introduced in  \eqref{mu-RC-k}. 
   
   \begin{proposition}\label{theorem-submain-bis}
   Assuming \eqref{eq-non-Shepp}, then for any integer $k\ge 1$, 
   \begin{align*}
  \big|\widehat{\mu_{\RC}^{\geqslant k}}(n)\big|^2=O(n^{-\tau})~\text{ a.s. for any}~\tau<1-D_\ell.
   \end{align*}
   In particular, conditioned on the event $\{\mu_\RC^{\geqslant k} \ne 0\}$, the uncovered set $K_\ell$ is almost surely Salem with Fourier dimension \[\dim_{\mathcal F} K_\ell = \dim_{\mathcal H} K_\ell = 1-D_\ell.\]
   \end{proposition}

\begin{proof}
Fix any integer $k\ge 1$.  Note that  $\mu_{\RC}^{\geqslant k}$ coincides with that of the standard multiplicative chaos measures of Dvoretzky random covering associated to the sequence 
\[
\ell': = (\ell_k, \ell_{k+1}, \ell_{k+2},\ldots).
\] 
Then, by the definition \eqref{def-D-l}, one has $D_\ell = D_{\ell'}$.  
Moreover, since $\mu_\RC^{\geqslant k}$ is constructed in \eqref{mu-RC-k}  by using the  collection $(I_{j}'(\omega))_{j=1}^\infty$ of random arcs given by 
\[
(I_1'(\omega), I_2'(\omega), \ldots) := (I_k(\omega), I_{k+1}(\omega), \ldots), 
\]
one has $K_\ell (\omega)= K_{\ell'}(\omega)$ by the definition \eqref{def-K-l}. Therefore, the proof of Proposition~\ref{theorem-submain-bis} follows verbatim that of Proposition~\ref{theorem-submain}. 
\end{proof}

\begin{proof}[Proof of Theorem~\ref{theorem-main}]
By Lemma \ref{lem-zero-one-law}, the assumption \eqref{eq-non-Shepp} implies  
\[
\PP\Big( \bigcup_{k\ge 1} \big\{\mu_{\RC}^{\geqslant k} \ne 0\big\}\Big) = 1. 
\]
Hence Theorem~\ref{theorem-main} follows immediately from Proposition~\ref{theorem-submain-bis}. 
\end{proof}

\appendix
\section*{Appendix}

\renewcommand{\thesubsection}{\Alph{subsection}}
\setcounter{subsection}{0}

	\subsection{Proof of the implication $\eqref{eq-non-Shepp} \Longrightarrow D_\ell\le 1$}\label{appendix:non-shepp-D_l<=1}
    Let $s_k:= \ell_1 + \cdots + \ell_k$. Assume that $D_\ell > 1$, there exists $\delta > 0$ and an infinite subsequence $\{s_{k_j}\}_{j\geq 1}$ such that $s_{k_j} > (1+\delta)\log k_j$ for all $j\geq 1$. 
    
    Since the sequence $s_k$ is increasing, we have $$\sum_{k=1}^\infty k^{-2} \exp(s_k) \geq \sum_{k=k_j}^{2k_j} k^{-2} \exp(s_k) \geq (k_j+1) (2k_j)^{-2} \exp(s_{k_j})\geq k_j^{\delta}/4 \longrightarrow \infty$$ as $j$ tends to the infinity. The proof is completed.
    \subsection{Proof of the Remark~\ref{rem-0dim}}\label{appendix:Kahane-remark}
		In the critical case $D_\ell=1$, we consider $I^{\lambda}_k(\omega):=(\omega_k,\omega_k+\lambda \ell_k)$ with $\lambda<1$, and let $K^{\lambda}_\ell$ be the uncovered set for the random arcs $I^{\lambda}_k$. Clearly, one has the monotone coupling $I_k^\lambda \subset I_k$ and hence the inverse monotone coupling $K_\ell\subset K^{\lambda}_{\ell}$. Then it follows from Kahane's theorem (for non-critical cases) that $$\dim_{\mathcal H} K_\ell \leq \dim_{\mathcal H} K^{\lambda}_{\ell}=1-\lambda.$$ The proof is completed by letting $\lambda \to 1$.

    \subsection{Proof of Remark~\ref{remark:Tan-Dl<1-verification}}\label{appendix:Tan-Dl<1-verification}
		Let $s_k:=\ell_1+\cdots +\ell_k$. Since $\limsup_{k\to \infty} (s_k/\log k)<1$, there exist $\delta>0$ and $K_0\in \mathbb N_+$, such that $s_k\leq (1-\delta)\log k$ for all  $k\geq K_0.$ By the Abel summation, for any $K\ge K_0$, 
		\begin{align*}
			\sum_{k=K_0}^K (\ell_k-\ell_{k+1}) \exp(s_k)=&\ell_{K_0} \exp(s_{K_0}) - \ell_{K+1} \exp(s_{K}) + \sum_{k=K_0+1}^K \ell_k\cdot \big[\exp(s_k) - \exp(s_{k-1})\big]\\
			\leq&  \ell_{K_0} \exp(s_{K_0}) + \sum_{k=K_0+1}^K \ell_k \cdot\big[\exp(\ell_k)-1\big]\cdot\exp(s_{k-1}).
		\end{align*}
		Since $\ell$ is non-increasing, we have $k\ell_k \leq s_k$ and hence $\ell_k \leq\frac{s_k}{k}\leq \frac{\log k}{k}$ for all  $k\geq K_0.$ Furthermore, we may assume, without loss of generality, that $K_0$ is sufficiently large so that $[\exp(\ell_k) - 1] \leq 2\ell_k$ for all $k \geq K_0$ (this is possible since $\ell_k \to 0$). Therefore, for any integer $K>K_0$, it holds that
		\begin{align*}
			\sum_{k=K_0}^K (\ell_k-\ell_{k+1}) \exp(s_k) \leq&\ell_{K_0} \exp(s_{K_0}) + \sum_{k=K_0+1}^K \ell_k \cdot\big[\exp(\ell_k)-1\big]\cdot\exp(s_{k-1})\\
			\leq& \ell_{K_0} \exp(s_{K_0}) + 2\sum_{k=K_0+1}^K \ell_k^2 \cdot \exp(s_{k-1})\\
			\leq& \frac{\log K_0}{K_0}\cdot K_0^{1-\delta}+ 2 \sum_{k=K_0+1}^K \Big(\frac{\log k}{k}\Big)^2\cdot (k-1)^{1-\delta}\\
			\leq& K_0^{-\delta} \log K_0+2 \sum_{k=K_0+1}^\infty k^{-1-\delta} \log k< \infty.
		\end{align*}
		The proof is completed.


\end{document}